\documentclass[11pt,reqno,draft]{amsart}

\usepackage{amssymb}

\newtheorem{theorem}{Theorem}[section]

\newtheorem{proposition}[theorem]{Proposition}

\newtheorem{lemma}[theorem]{Lemma}

\newtheorem{corollary}[theorem]{Corollary}

\DeclareMathOperator{\Aut}{Aut}

\DeclareMathOperator{\End}{End}

\DeclareMathOperator{\F}{F}

\DeclareMathOperator{\Lat}{L}

\DeclareMathOperator{\var}{var}

\DeclareMathOperator{\ZR}{ZR}

\makeatletter

\renewcommand*\subjclass[2][2000]{\def\@subjclass{#2}\@ifundefined
{subjclassname@#1}{\ClassWarning{\@classname}{Unknown edition (#1) of
Mathematics Subject Classification; using '2000'.}}{\@xp\let\@xp
\subjclassname\csname subjclassname@#1\endcsname}}

\makeatother

\begin{document}

\title[Modular and lower-modular elements of lattices of varieties]{Modular
and lower-modular elements\\
of lattices of semigroup varieties}

\author[V. Yu. Shaprynski\v{\i}]{V. Yu. Shaprynski\v{\i}\\
\\
Communicated by L. N. Shevrin}

\address{Department of Mathematics and Mechanics, Ural State University,
Lenina 51, 620083 Ekaterinburg, Russia}

\email{vshapr@yandex.ru}

\thanks{The work was partially supported by the Russian Foundation for Basic
Research (grants No.~09-01-12142,~10-01-00524) and the Federal
Education Agency of the Russian Federation (project
No.~2.1.1/3537).}

\keywords{Semigroup, variety, lattice of varieties, commutative variety,
modular element, lower-modular element}

\subjclass{Primary 20M07, secondary 08B15}

\begin{abstract}
The paper contains three main results. First, we show that if a commutative
semigroup variety is a modular element of the lattice \textbf{Com} of all
commutative semigroup varieties then it is either the variety $\mathcal{COM}$
of all commutative semigroups or a nil-variety or the join of a nil-variety
with the variety of semilattices. Second, we prove that if a commutative
nil-variety is a modular element of \textbf{Com} then it may be given within
$\mathcal{COM}$ by 0-reduced and substitutive identities only. Third, we
completely classify all lower-modular elements of \textbf{Com}. As a
corollary, we prove that an element of \textbf{Com} is modular whenever it is
lower-modular. All these results are precise analogues of results concerning
modular and lower-modular elements of the lattice of all semigroup varieties
obtained earlier by Jezek, McKenzie, Vernikov, and the author. As an
application of a technique developed in this paper, we provide new proofs of
the `prototypes' of the first and the third our results.
\end{abstract}

\maketitle

\section{Introduction and summary}
\label{intr}

There are several papers where special elements of the lattice of all
semigroup varieties were examined. Results in this area obtained before 2009
have been overviewed in Section~14 of the survey~\cite
{Shevrin-Vernikov-Volkov-09}. Recall the definitions of special elements
mentioned in the present paper. An element $x$ of a lattice $\langle L;\vee,
\wedge\rangle$ is called
\begin{align*}
&\text{\emph{modular} if}\quad&&\forall\,y,z\in L\colon\quad y\le z
\longrightarrow(x\vee y)\wedge z=(x\wedge z)\vee y,\\
&\text{\emph{lower-modular} if}\quad&&\forall\,y,z\in L\colon\quad x\le y
\longrightarrow x\vee(y\wedge z)=y\wedge(x\vee z),\\
&\text{\emph{distributive} if}\quad&&\forall\,y,z\in L\colon\quad x\vee(y
\wedge z)=(x\vee y)\wedge(x\vee z),
\end{align*}
and \emph{neutral} if, for all $y,z\in L$, the sublattice generated by the
elements $x$, $y$, and $z$ is distributive. \emph{Upper-modular} and \emph
{codistributive} elements are defined dually to lower-modular and
distributive ones respectively. It is evident that every [co]distributive
element is lower-modular [upper-modular] and that every neutral element is
distributive, codistributive, and modular.

The lattice of all semigroup varieties is denoted by \textbf{SEM}. Neutral
elements of this lattice were determined by Volkov in~\cite{Volkov-05}. A
description of distributive and lower-modular elements in \textbf{SEM} was
obtained by Vernikov and the author in~\cite{Vernikov-Shaprynskii-10}
and~\cite{Shaprynskii-Vernikov-lmod} respectively. An essential information
about modular elements of \textbf{SEM} was found in~\cite{Jezek-McKenzie-93,
Vernikov-07-mod}. Some of the mentioned results are interesting from the
point of view of this paper. To formulate these results, we need some
definitions and notation.

For convenience, we call a semigroup variety \emph{modular} if it is a
modular element of the lattice \textbf{SEM}, and adopt analogous agreement
for all other types of special elements. As usual, we replace a pair of
identities $wx=xw=w$ where the letter $x$ does not occur in the word $w$ by
the symbolic identity $w=0$. Identities of the form $w=0$ as well as
varieties given by such identities are called 0-\emph{reduced}. An identity
$u=v$ is called \emph{substitutive} if the words $u$ and $v$ depend on the
same letters and the word $v$ may be obtained from $u$ by renaming of
letters. Recall that a semigroup variety is called a \emph{nil-variety} if it
consists of nil-semigroups (or, equivalently, if it satisfies an identity of
the form $x^n=0$ for some $n$). It is evident that every 0-reduced variety is
a nil-variety. By $\mathcal T$, $\mathcal{SL}$, and $\mathcal{SEM}$ we denote
the trivial variety, the variety of all semilattices, and the variety of all
semigroups respectively.

\begin{proposition}
\label{SEM mod nec}
If a semigroup variety $\mathcal V$ is modular then either $\mathcal{V=SEM}$
or $\mathcal{V=M\vee N}$ where $\mathcal M$ is one of the varieties $\mathcal
T$ or $\mathcal{SL}$, while $\mathcal N$ is a nil-variety.\qed
\end{proposition}

This fact was proved (in slightly weaker form and some other terminology)
in~\cite{Jezek-McKenzie-93}, Proposition 1.6; a deduction of Proposition~\ref
{SEM mod nec} from~\cite[Proposition~1.6]{Jezek-McKenzie-93} is given
explicitly in~\cite[Proposition~2.1]{Vernikov-07-mod}. Proposition~\ref
{SEM mod nec}, together with Lemma~\ref{join with SL} formulated below,
completely reduces the problem of description of modular varieties to the
nil-case. The following necessary condition for a nil-variety to be modular
(stronger than one given by Proposition~\ref{SEM mod nec}) is true.

\begin{proposition}[{\cite[Theorem~2.5]{Vernikov-07-mod}}]
\label{SEM mod nil-nec}
If a nil-variety of semigroups is modular then it may be given by $0$-reduced
and substitutive identities only.\qed
\end{proposition}

The last earlier result we cite here is the following

\begin{proposition}[{\cite[Theorem~1.1]{Shaprynskii-Vernikov-lmod}}]
\label{SEM lmod}
A semigroup variety $\mathcal V$ is lower-modular if and only if either
$\mathcal{V=SEM}$ or $\mathcal{V=M\vee N}$ where $\mathcal M$ is one of the
varieties $\mathcal T$ or $\mathcal{SL}$, while $\mathcal N$ is a $0$-reduced
variety.\qed{\sloppy

}
\end{proposition}

It is natural to study special elements not only in the whole lattice \textbf
{SEM} but also in its most important sublattices. One of such sublattices is
the lattice \textbf{Com} of all commutative semigroup varieties. This lattice
is intensively studied and several deep results were obtained here. The
lattice $\mathbf{Com}$ contains an isomorphic copy of any finite lattice
(this follows from results of the papers~\cite{Burris-Nelson-71} and~\cite
{Pudlak-Tuma-80}), so it does not satisfy any non-trivial lattice identity.
On the other hand, this lattice is countably infinite~\cite{Perkins-69}. Some
parametrization of the lattice \textbf{Com} was suggested in~\cite
{Kisielewicz-94}. The study of special elements in \textbf{Com} was started
by the author in~\cite{Shaprynskii-dn} where distributive and neutral
elements in this lattice were described.

Main results of this paper are precise analogoues of Propositions~\ref
{SEM mod nec}--\ref{SEM lmod} for the lattice \textbf{Com}. To formulate
these results, we need some additional definitions and notation. We call a
commutative semigroup variety \emph{modular in} $\mathbf{Com}$ if it is a
modular element of the lattice \textbf{Com}, and adopt analogous agreement
for all other types of special elements. By $\mathcal{COM}$ we denote the
variety of all commutative semigroups. A commutative semigroup variety is
called 0-\emph{reduced in} \textbf{Com} if it is defined within $\mathcal
{COM}$ by 0-reduced identities only. Our main results are the following three
theorems.

\begin{theorem}
\label{Com mod nec}
If a commutative semigroup variety $\mathcal V$ is modular in $\mathbf{Com}$
then either $\mathcal{V=COM}$ or $\mathcal{V=M\vee N}$ where $\mathcal M$ is
one of the varieties $\mathcal T$ or $\mathcal{SL}$, while $\mathcal N$ is a
nil-variety.
\end{theorem}

\begin{theorem}
\label{Com mod nil-nec}
If a commutative nil-variety of semigroups is modular in $\mathbf{Com}$ then
it may be given within the variety $\mathcal{COM}$ by $0$-reduced and
substitutive identities only.
\end{theorem}

\begin{theorem}
\label{Com lmod}
A commutative semigroup variety $\mathcal V$ is lower-modular in $\mathbf
{Com}$ if and only if either $\mathcal{V=COM}$ or $\mathcal{V=M\vee N}$ where
$\mathcal M$ is one of the varieties $\mathcal T$ or $\mathcal{SL}$, while
$\mathcal N$ is a $0$-reduced in $\mathbf{Com}$ variety.
\end{theorem}

The technique developed in the course of proving these theorems permits to
give new proofs of Propositions~\ref{SEM mod nec} and~\ref{SEM lmod}.
Moreover, in fact Propositions~\ref{SEM mod nec} and~\ref{SEM lmod} may be
verified by practically the same arguments as Theorems~\ref{Com mod nec}
and~\ref{Com lmod} respectively. It seems for us that these new proofs of
earlier results are of some independent interest, and we include these proofs
in the paper.

The article consists of four sections. Section~\ref{prel} contains some
auxuliary results. In Section~\ref{proof} we prove Theorems~\ref
{Com mod nec}--\ref{Com lmod} and provide new proofs of Propositions~\ref
{SEM mod nec} and~\ref{SEM lmod}. We deduce all these results from some
general assertion about modular and lower-modular elements of the subvariety
lattice of an arbitrary overcommutative variety (see Proposition~\ref
{periodic in oc}). Finally, in Section~\ref{corol} we provide some
corollaries of Theorem~\ref{Com lmod}.

\section{Preliminaries}
\label{prel}

We start with some definitions and auxiliary results. Suppose that $I$ is a
lattice identity of the form
$$s(x_0,x_1,\dots,x_n)=t(x_0,x_1,\dots,x_n)\ldotp$$
We say that an element $x$ of a lattice $L$ is an $I$-\emph{element} of this
lattice if
$$\forall\,x_1,\dots,x_n\in L\colon\quad s(x,x_1,\dots,x_n)=t(x,x_1,\dots,
x_n)\ldotp$$
Note that all types of elements mentioned above are partial cases of
$I$-elements. For distributive and codistributive elements this is evident.
It is well-known that an element $x\in L$ is neutral if and only if
$$\forall\,y,z\in L\colon\quad (x\vee y)\wedge(y\vee z)\wedge(z\vee x)=(x
\wedge y)\vee(y\wedge z)\vee(z\wedge x)$$
(see Theorem~III.2.4. in~\cite{Gratzer-98}, for instance). Finally, modular
elements can be defined by the condition
$$\forall y,z\in L\colon\quad(x\vee y)\wedge(y\vee z)=(x\wedge(y\vee z))\vee
y,$$
lower-modular ones by the condition
$$\forall y,z\in L\colon\quad(x\vee y)\wedge(x\vee z)=((x\vee y)\wedge z)\vee
x,$$
and upper-modular ones by the condition dual to the latter one.

The following lemma can be obtained by a combination of Corollary~2.1 and
Lemma~2.4 of~\cite{Shaprynskii-dn}.

\begin{lemma}
\label{join with SL}
Let $I$ be a lattice identity satisfied by distributive lattices. A \textup
[commutative\textup] semigroup variety $\mathcal V$ is an $I$-element of the
lattice $\mathbf{SEM}$ \textup[respectively $\mathbf{Com}$\textup] if and
only if the variety $\mathcal{V\vee SL}$ has this property.\qed
\end{lemma}

\begin{lemma}
\label{0-red is mod}
Every $0$-reduced \textup[in $\mathbf{Com}$\textup] semigroup variety is
modular \textup[in $\mathbf{Com}$\textup].\qed
\end{lemma}

The fact that a 0-reduced variety is modular was noted for the first time
in~\cite{Vernikov-Volkov-88}, Corollary~3, and rediscovered in some other
terms in~\cite{Jezek-McKenzie-93}, Proposition~1.1. The `commutative part' of
Lemma~\ref{0-red is mod} was verified in~\cite{Shaprynskii-dn},
Proposition~2.1. Note that, in fact, Lemma~\ref{0-red is mod} readily follows
from~\cite{Jezek-81}, Proposition~2.2.

\begin{lemma}
\label{lmod nil}
A \textup[commutative\textup] nil-variety of semigroups is lower-modular
\textup[in $\mathbf{Com}$\textup] if and only if it is $0$-reduced \textup[in
$\mathbf{Com}$\textup].\qed
\end{lemma}

For lower-modular varieties, this fact was proved in~\cite
{Vernikov-Volkov-88}, Corollary~3 (the `if' part) and~\cite
{Vernikov-07-lmod}, Corollary~2.7 (the `only if' part); for lower-modular in
\textbf{Com} varieties it was verified in~\cite{Shaprynskii-dn},
Proposition~2.2.

\begin{lemma}[{\cite[Proposition~2.3]{Shaprynskii-dn}}]
\label{umod 0-red}
A $0$-reduced in $\mathbf{Com}$ semigroup variety is upper-modular in
$\mathbf{Com}$ if and only if it satisfies the identity $x^2y=0$.\qed
\end{lemma}

If $u$ is a word and $x$ is a letter then we denote by $c(u)$ the \emph
{content} of $u$, that is, the set of all letters occurring in $u$; further,
$h(u)$ [respectively $t(u)$] denotes the first [the last] letter of $u$, and
$\ell_x(u)$ is the number of occurences of $x$ in $u$. The symbol $\equiv$
stands for the equality relation on the free semigroup of a countably
infinite rank. We denote by $\var\Sigma$ the semigroup variety given by the
identity system $\Sigma$. Put
\begin{align*}
&\mathcal{LZ}=\var\{xy=x\},\\
&\mathcal{RZ}=\var\{xy=y\},\\
&\mathcal P=\var\{xy=x^2y,\,x^2y^2=y^2x^2\},\\
&\overleftarrow{\mathcal P}=\var\{xy=xy^2,\,x^2y^2=y^2x^2\}\ldotp
\end{align*}
The first statement of the following lemma is evident, while the second one
is verified in~\cite{Golubov-Sapir-82}, Lemma~7.

\begin{lemma}
\label{word problem}
The identity $u=v$ holds:
\begin{itemize}
\item[\textup{(i)}]in the variety $\mathcal{LZ}$ if and only if $h(u)\equiv h
(v)$;
\item[\textup{(ii)}]in the variety $\mathcal P$ if and only if $c(u)=c(v)$
and either $\ell_{t(u)}(u)>1$ and $\ell_{t(v)}(v)>1$ or $\ell_{t(u)}(u)=
\ell_{t(v)}(v)=1$ and $t(u)\equiv t(v)$.\qed
\end{itemize}
\end{lemma}

A semigroup variety is called \emph{periodic} if it consists of periodic
semigroups. Lemma~2 of the paper~\cite{Volkov-89} and the proof of
Proposition~1 of the same paper imply the following

\begin{lemma}
\label{M+N}
If a periodic semigroup variety $\mathcal V$ does not contain the varieties
$\mathcal{LZ}$, $\mathcal{RZ}$, $\mathcal P$, and $\overleftarrow{\mathcal
P}$ then $\mathcal{V=M\vee N}$ where the variety $\mathcal M$ is generated by
a monoid, while $\mathcal N$ is a nil-variety.\qed
\end{lemma}

Let $\mathcal V$ be a commutative nil-variety of semigroups. We denote by
$\ZR(\mathcal V)$ the variety given by the commutative law and all 0-reduced
identities that hold in $\mathcal V$. It is clear that $\ZR(\mathcal V)$ is
the least 0-reduced in \textbf{Com} variety that contains $\mathcal V$. For
any natural $n$ we denote by $\mathcal A_n$ the variety of periodic Abelian
groups whose exponent divides $n$.

\begin{lemma}[{\cite[Lemma~2.5]{Shaprynskii-dn}}]
\label{V+G=ZR(V)+G}
If a commutative semigroup variety $\mathcal V$ satisfies the identity $x^n=
0$ then $\mathcal{V\vee A}_n=\ZR(\mathcal{V)\vee A}_n$.\qed
\end{lemma}

We need the following two well known and easily verified technical remarks
about identities of nilsemigroups.

\begin{lemma}
\label{split}
Let $\mathcal V$ be a nil-variety of semigroups.
\begin{itemize}
\item[\textup{(i)}]If the variety $\mathcal V$ satisfies an identity $u=v$
with $c(u)\ne c(v)$ then $\mathcal V$ satisfies also the identity $u=0$.
\item[\textup{(ii)}]If the variety $\mathcal V$ satisfies an identity of the
form $u=vuw$ where at least one the words $v,w$ is non-empty then $\mathcal
V$ satisfies also the identity $u=0$.\qed
\end{itemize}
\end{lemma}

A semigroup variety $\mathcal V$ is called \emph{overcommutative} if
$\mathcal{V\supseteq COM}$. An identity $u=v$ is called \emph{balanced} if
$\ell_x(u)=\ell_x(v)$ for every letter $x$. For convenience of references, we
formulate the following two generally known facts.

\begin{lemma}
\label{periodic or oc}
An arbitrary semigroup variety is either periodic or overcommutative.\qed
\end{lemma}

\begin{lemma}
\label{oc ident}
If an overcommutative semigroup variety satisfies some identity then this
identity is balanced.\qed
\end{lemma}

\section{Proofs of main results}
\label{proof}

We need some new definitions and notation. We denote by $\F$ the free
semigroup over a countably infinite alphabet. If $u\in\F$ then we denote by
$\ell(u)$ the length of $u$. For an arbitrary semigroup variety $\mathcal X$,
we denote by $\Lat(\mathcal X)$ its subvariety lattice, by $\F(\mathcal X)$
its free semigroup over a countably infinite alphabet, by $\F_c(\mathcal X)$
its free semigroup over an alphabet $c$, and by $\Lat(\F(\mathcal X))$ the
lattice of fully invariant congruences on $\F(\mathcal X)$. The equality
relation on $\F(\mathcal X)$ and $\F_c(\mathcal X)$ will be denoted by
$\equiv$. Let now $\mathcal X$ be an overcommutative variety and $U\in\F
(\mathcal X)$. Lemma~\ref{oc ident} permits to define the \emph{length} of
$U$ (denoted by $\ell(U)$) as the length of an arbitrary word $u\in U$, the
\emph{content} of $U$ (denoted by $c(U)$) as the content of an arbitrary word
$u\in U$, and the \emph{number of occurerences of a letter $x$ in $U$}
(denoted by $\ell_x(U)$) as the number of occurences of $x$ in an arbitrary
word $u\in U$. The $\mathcal X$-\emph{image} of an arbitrary word is its
image under the natural homomorphism from $\F$ to $\F(\mathcal X)$. We call
elements $U$ and $V$ of $\F(\mathcal X)$ \emph{equivalent} if $U\equiv\xi(V)$
for some automorphism $\xi$ on $\F(\mathcal X)$ (in other words, if $|c(U)|=
|c(V)|$ and $U\equiv\xi(V)$ for some isomorphism $\xi$ from $\F_{c(V)}
(\mathcal X)$ to $\F_{c(U)}(\mathcal X)$). We call two words $\mathcal
X$-\emph{equivalent} if their $\mathcal X$-images are equivalent. We say that
an element $W\in\F(\mathcal X)$ is $\mathcal X$-\emph{unstable} if $\xi(W)
\not\equiv W$ for any non-trivial automorphism $\xi$ on $\F_{c(W)}(\mathcal
X)$. Otherwise the element $W$ is called $\mathcal X$-\emph{stable}. A word
is called $\mathcal X$-[\emph{un}]\emph{stable} if its $\mathcal X$-image is
[un]stable. The group of automorphisms [the semigroup of endomorphisms] of a
semigroup $S$ is denoted by $\Aut(S)$ [respectively $\End(S)$].

The following lemma plays the key role in what follows.

\begin{lemma}
\label{u=v,s=t->u=s}
Let $\mathcal X$ be an overcommutative semigroup variety and let $u,v,s$,
and $t$ be $\mathcal X$-unstable and pairwise $\mathcal X$-non-equivalent
words with the same length and the same content. If a semigroup variety
$\mathcal V$ is either a modular or a lower-modular element of the lattice
$\Lat(\mathcal X)$ and $\mathcal V$ satisfies the identities $u=v$ and $s=t$
then $\mathcal V$ satisfies also the identity $u=s$.
\end{lemma}

\begin{proof}
Let $\ell$ be the length of the words $u,v,s$, and $t$, and let $c$ be the
content of these words. Denote by $U,V,S$, and $T$ the $\mathcal X$-images of
the words $u,v,s$, and $t$ respectively. Let $\alpha$ be the fully invariant
congruence on $\F(\mathcal X)$ corresponding to the variety $\mathcal V$.
Then $U\,\alpha\,V$ and $S\,\alpha\,T$. We aim to prove that $U\,\alpha\,S$.

Consider any two distinct elements $A,B\in\{U,V,S,T\}$ and any two
automorphisms $\xi$ and $\zeta$ on $\F(\mathcal X)$. We shall prove that the
pairs $\{\xi(A),\xi(B)\}$ and $\{\zeta(A),\zeta(B)\}$ either coincide or do
not intersect. Suppose that
$$\{\xi(A),\xi(B)\}\cap\{\zeta(A),\zeta(B)\}\ne\varnothing\ldotp$$
Acting by the automorphism $\xi^{-1}$ we get $\{A,B\}\cap\{\varphi(A),\varphi
(B)\}\neq\varnothing$ where $\varphi=\xi^{-1}\zeta$. Since the elements $A$
and $B$ are not equivalent, $A\not\equiv\varphi(B)$ and $B\not\equiv\varphi
(A)$, so either $A\equiv\varphi(A)$ or $B\equiv\varphi(B)$. We may suppose
without loss of generality that $A\equiv\varphi(A)$. Since $\varphi$ is an
automorphism, the restriction $\overline{\varphi}$ of $\varphi$ on $\F_{c(A)}
(\mathcal X)$ is an isomorphism from $\F_{c(A)}(\mathcal X)$ to
$\F_{\varphi(c(A))}(\mathcal X)$. But $\varphi(c(A))=c(\varphi(A))=c(A)$, so
$\overline{\varphi}$ is an automorphism on $\F_{c(A)}(\mathcal X)$. Since $A$
is unstable, the automorphism $\overline{\varphi}$ is trivial. Further, $c
(A)=c(B)$, whence $\varphi(B)\equiv\overline{\varphi}(B)\equiv B$. Acting by
the automorphism $\xi$ on the equality $\varphi(B)\equiv B$, we get $\xi(B)
\equiv\zeta (B)$, whence $\{\xi(A),\xi(B)\}=\{\zeta(A),\zeta(B)\}$.

Further, consider the set
$$\mathfrak W=\{W\in\F(\mathcal X)\mid\text{either}\ \ell(W)>\ell\ \text{or}\
\ell(W)=\ell\ \text{and}\ |c(W)|<|c|\}\ldotp$$
Since an arbitrary automorphism $\xi$ on $\F(\mathcal X)$ does not change the
parameters $\ell(W)$ and $|c(W)|$ of an element $W\in\F(\mathcal X)$, none of
the elements $\xi(U),\xi(V),\xi(S)$, and $\xi(T)$ belongs to $\mathfrak W$.
Therefore none of the pairs $\{\xi(U),\xi(S)\}$ and $\{\xi(V),\xi(T)\}$
intersects with $\mathfrak W$. Further, let $\xi,\zeta\in\Aut(\F(\mathcal
X))$. Then $\{U,S\}\cap\{\xi^{-1}\zeta(V),\xi^{-1}\zeta(T)\}=\varnothing$
because the elements $U,V,S$, and $T$ are pairwise non-equivalent. Therefore
$\{\xi(U),\xi(S)\}\cap\{\zeta(V),\zeta(T)\}=\varnothing$. Combining these
observations with the arguments from the previous paragraph, we obtain that
there is a partition of the set $\F(\mathcal X)$ whose non-singleton classes
are $\mathfrak W$ and 2-element sets of the form $\{\xi(U),\xi(S)\}$ and
$\{\xi(V),\xi(T)\}$ where $\xi$ runs over $\Aut(\F(\mathcal X))$. Denote the
equivalence relation corresponding to this partition by $\gamma$. By the same
arguments, there are partitions of $\F(\mathcal X)$ whose non-singleton
classes are:{\sloppy

}
\begin{itemize}
\item[$\bullet$]$\mathfrak W$ and pairs $\{\xi(V),\xi(T)\}$ where $\xi$ runs
over $\Aut(\F(\mathcal X))$;
\item[$\bullet$]$\mathfrak W$ and pairs $\{\xi(U),\xi(V)\}$ where $\xi$ runs
over $\Aut(\F(\mathcal X))$.
\end{itemize}
Denote the equivalence relation corresponding to the former [the latter] of
these partitions by $\beta$ [respectively $\delta$]. Obviously, $\beta
\subseteq\gamma$.

Now we aim to check that $\gamma$ is a fully invariant congruence. By the
definition of $\gamma$ the condition $W_1\,\gamma\,W_2$ for distinct elements
$W_1,W_2\in\F(\mathcal X)$ implies either $W_1,W_2\in\mathfrak W$ or $\{W_1,
W_2\}=\{\xi(U),\xi(S)\}$ or $\{W_1,W_2\}=\{\xi(V),\xi(T)\}$ for some
automorphism $\xi$. We see that $\ell(W_1),\ell(W_2)\ge\ell$ always.
Therefore if $W\in\F(\mathcal X)$ then $\ell(W_1W),\ell(W_2W),\ell(WW_1),\ell
(WW_2)>\ell$, whence $W_1W,W_2W,WW_1,WW_2\in\mathfrak W$. We see that $W_1W\,
\gamma\,W_2W$ and $WW_1\,\gamma\,WW_2$ for any element $W$. This means that
the relation $\gamma$ is a congruence. Now we shall prove that $\gamma$ is
invariant under an arbitrary endomorphism $\varphi$ on $\F(\mathcal X)$.
Obviously, for any $W\in\F(\mathcal X)$ we have $\ell(\varphi(W))\ge\ell(W)$
and if $\ell(\varphi(W))=\ell(W)$ (i.~e., if $\varphi$ maps every letter of
$c(W)$ to a letter) then $|c(\varphi(W))|\le|c(W)|$. This implies that
$\varphi(\mathfrak W)\subseteq\mathfrak W$. It remains to prove that $\varphi
(W_1)\,\gamma\,\varphi(W_2)$ for any distinct elements $W_1,W_2\in\F(\mathcal
X)$ with $W_1\,\gamma\,W_2$ and $W_1,W_2\notin\mathfrak W$. The latest means
that either $\{W_1,W_2\}=\{\xi(U),\xi(S)\}$ or $\{W_1,W_2\}=\{\xi(V),\xi
(T)\}$ for some $\xi$. We may suppose without loss of generality that $W_1
\equiv\xi(U)$ and $W_2\equiv\xi(S)$. The following three cases are possible:
\begin{itemize}
\item[1)]$\varphi$ maps some letter of $c(\xi(U))$ to an element of length
$>1$;
\item[2)]$\varphi$ maps every letter of $c(\xi(U))$ to a letter and maps some
two distinct letters of $c(\xi(U))$ to the same letter;
\item[3)]$\varphi$ maps all letters of $c(\xi(U))$ to distinct letters.
\end{itemize}
In the case 1) we have $\ell(\varphi(W_1)),\ell(\varphi(W_2))>\ell$, while in
the case 2) we have $\ell(\varphi(W_1))=\ell(\varphi(W_2))=\ell$ and $|c
(\varphi(W_1))|=|c(\varphi(W_2))|<|c|$. In both the cases $\varphi(W_1),
\varphi(W_2)\in\mathfrak W$. In the case 3) there is an automorphism
$\overline{\varphi}$ on $\F(\mathcal X)$ such that $\overline{\varphi}(W_1)
\equiv\varphi(W_1)$ and $\overline{\varphi}(W_2)\equiv\varphi(W_2)$, so the
pair $\{\varphi(W_1),\varphi(W_2)\}=\{\overline{\varphi}(\xi(U)),\overline
{\varphi}(\xi(S))\}$ is a $\gamma$-class. In all three cases $\varphi(W_1)\,
\gamma\, \varphi(W_2)$. We have verified that $\gamma$ is a fully invariant
congruence. Analogous arguments show that $\beta$ and $\delta$ are fully
invariant congruences too.

Suppose that the variety $\mathcal V$ is a modular element of the lattice
$\Lat(\mathcal X)$. Then $\alpha$ is a modular element of the lattice $\Lat
(\F(\mathcal X))$. Since $\beta\subseteq\gamma$, we have $(\alpha\vee\beta)
\wedge\gamma=(\alpha\wedge\gamma)\vee\beta$. Further, $U\,\alpha\,V\,\beta\,
T\,\alpha\,S$ and $U\,\gamma\,S$, whence $(U,S)\in(\alpha\vee\beta)\wedge
\gamma=(\alpha\wedge\gamma)\vee\beta$. On the other hand, $U\notin\mathfrak
W$ and $U\notin\{\xi(V),\xi(T)\}$ for any $\xi\in\Aut(\F(\mathcal X))$, so
the set $\{U\}$ is a $\beta$-class. If $\{U\}$ is an $(\alpha\wedge
\gamma)$-class then $\{U\}$ is an $((\alpha\wedge\gamma)\vee\beta)$-class
too. But this fails because $(U,S)\in(\alpha\wedge\gamma)\vee\beta$. Thus
there is an element $R$ such that $R\not\equiv U$ and $(U,R)\in\alpha\wedge
\gamma$. In particular $U\,\gamma\,R$. By the definition of $\gamma$ the set
$\{U,S\}$ is a $\gamma$-class. Hence $R\equiv S$. We see that $(U,S)\in\alpha
\wedge\gamma$. In particular, $U\,\alpha\,S$, and we are done.

We have verified the `modular half' of our lemma. Suppose now that the
variety $\mathcal V$ is a lower-modular element of the lattice $\Lat(\mathcal
X)$. Then $\alpha$ is an upper-modular element of the lattice $\Lat(\F
(\mathcal X))$. Put $\rho=\delta\wedge\alpha$. Since $\rho\subseteq\alpha$,
we have $(\gamma\vee\rho)\wedge\alpha=(\gamma\wedge\alpha)\vee\rho$. Since
$U\,\delta\,V$ and $U\,\alpha\,V$, we have $U\,\rho\,V$. Thus $S\,\gamma\,U\,
\rho\,V\,\gamma\,T$ and $S\,\alpha\,T$, whence $(S,T)\in(\gamma\vee\rho)
\wedge\alpha=(\gamma\wedge\alpha)\vee\rho$. On the other hand, $S\notin
\mathfrak W$ and $S\notin\{\xi(U),\xi(V)\}$ for any automorphism $\xi$, so
the set $\{S\}$ is a $\delta$-class. Hence $\{S\}$ is a $\rho$-class because
$\rho\subseteq\delta$. If $\{S\}$ is a $(\gamma\wedge\alpha)$-class then
$\{S\}$ is a $((\gamma\wedge\alpha)\vee\rho)$-class too. But this fails
because $(S,T)\in(\gamma\wedge\alpha)\vee\rho$. Thus there is an element $Q$
such that $Q\not\equiv S$ and $(S,Q)\in\gamma\wedge\alpha$. In particular
$S\,\gamma\,Q$. By the definition of $\gamma$ the set $\{S,U\}$ is a
$\gamma$-class. Hence $Q\equiv U$. We see that $(S,U)\in\gamma\wedge\alpha$.
In particular, $S\,\alpha\,U$, and we are done.
\end{proof}

A semigroup variety $\mathcal V$ is called \emph{proper} if $\mathcal{V\ne
SEM}$. For any word $w\equiv x_{j_1}x_{j_2}\cdots x_{j_n}$ where $x_{j_1},
x_{j_2},\dots,x_{j_n}$ are (not necessarily dufferent) letters and for $1\le
i\le n$ we put $w[i]\equiv x_{j_i}$.

\begin{corollary}
\label{mod and lmod are periodic}
If\/ $\mathcal V$ is either a modular or a lower-modular proper semigroup
variety then $\mathcal V$ is periodic.
\end{corollary}

\begin{proof}
Since the variety $\mathcal V$ is proper, it satisfies some non-trivial
identity $w_1=w_2$. Suppose that $\mathcal V$ is not periodic. Then $\mathcal
V$ is overcommutative by Lemma~\ref{periodic or oc}, whence the identity
$w_1=w_2$ is balanced by Lemma~\ref{oc ident}. There is some $i\in\{1,2,
\dots,\ell(w_1)\}$ with $w_1[i]\not\equiv w_2[i]$. Put $x\equiv w_1[i]$ and
$y\equiv w_2[i]$. Consider the words $u\equiv x^2w_1$, $v\equiv x^2w_2$, $s
\equiv xyw_1$, and $t\equiv xyw_2$. The identities $u=v$ and $s=t$ are
satisfied in $\mathcal V$ because they follow from $w_1=w_2$. The words $u$,
$v$, $s$, and $t$ have the same length and the same content because $\ell
(w_1)=\ell(w_2)$ and $c(w_1)=c(w_2)$. Now we aim to prove that these words
are pairwise non-equivalent. Suppose that $u$ and $v$ are equivalent, so $v
\equiv\xi(u)$ for some $\xi\in\Aut(\F)$. Then $x^2w_2\equiv(\xi(x))^2\xi
(w_1)$, whence $\xi(x)\equiv x$ and $\xi(w_1)\equiv w_2$. But $\xi(w_1)\equiv
w_2$ implies $\xi(x)\equiv\xi(w_1[i])\equiv w_2[i]\equiv y\not\equiv x$. The
contradiction shows that $u$ and $v$ are non-equivalent. The words $s$ and
$t$ are non-equivalent by analogous arguments. Finally, each of the words $u$
and $v$ is not equivalent to each of the words $s$ and $t$ because any of the
equalities $u\equiv\xi(s)$, $v\equiv\xi(s)$, $u\equiv\xi(t)$, and $v\equiv\xi
(t)$ for any $\xi\in\Aut(\F)$ would imply $\xi(xy)\equiv x^2$ that is
impossible. Since all words (in particular, the words $u$, $v$, $s$, and $t$)
are $\mathcal{SEM}$-unstable, we can apply Lemma~\ref{u=v,s=t->u=s} with
$\mathcal{X=SEM}$ and conclude that $\mathcal V$ satisfies the identity $u=
s$. This identity is not balanced because $\ell_x(u)=\ell_x(w_1)+2$, while
$\ell_x(s)=\ell_x(w_1)+1$. Lemmas~\ref{periodic or oc} and~\ref{oc ident}
imply that $\mathcal V$ is periodic.
\end{proof}

Note that the fact that a proper lower-modular semigroup variety is periodic
was verified earlier by another way in~\cite{Vernikov-07-lmod}, Theorem~1.

\begin{proposition}
\label{periodic in oc}
Let $\mathcal X$ be an overcommutative semigroup variety and $\mathcal V$ a
periodic subvariety of $\mathcal X$. If\/ $\mathcal V$ is either a modular or
a lower-modular element of the lattice $\Lat(\mathcal X)$ then $\mathcal{V=M
\vee N}$ where $\mathcal M$ is one of the varieties $\mathcal T$ or $\mathcal
{SL}$, while $\mathcal N$ is a nil-variety.
\end{proposition}

\begin{proof}
At first, we shall prove that $\mathcal{V=M\vee N}$ where $\mathcal M$ is
generated by a commutative monoid, while $\mathcal N$ is a nil-variety. Being
periodic, the variety $\mathcal V$ satisfies the identity $x^n=x^{n+m}$ for
some natural $n$ and $m$. We may assume without loss of generality that $n>
1$. Put
\begin{align*}
&u\equiv xy^{n+4m+3}z^nq,\phantom{x,}\quad v\equiv xy^{n+2m+3}z^{n+2m}q,\\
&s\equiv qy^{n+4m+2}z^{n+1}x,\quad t\equiv qy^{n+2m+2}z^{n+2m+1}x\ldotp
\end{align*}
The variety $\mathcal V$ satisfies the identities $u=v$ and $s=t$. Further
considerations are naturally divided into two cases.

\smallskip

\emph{Case}~1: the words $u,v,s$, and $t$ are $\mathcal X$-unstable. These
words have the same length and the same content. Besides that, these words
are pairwise $\mathcal X$-non-equivalent because if $w_1,w_2\in\{u,v,s,t\}$
and $w_1\not\equiv w_2$ then $\max\limits_{a\in c(w_1)}\{\ell_a(w_1)\}\ne\max
\limits_{a\in c(w_2)}\{\ell_a(w_2)\}$. Thus, we can apply Lemma~\ref
{u=v,s=t->u=s} and conclude that the identity $u=s$, that is, the identity
\begin{equation}
\label{u=s}
xy^{n+4m+3}z^nq=qy^{n+4m+2}z^{n+1}x
\end{equation}
holds in $\mathcal V$. Lemma~\ref{word problem} and its dual imply that this
identity fails in the varieties $\mathcal{LZ,RZ,P}$, and $\overleftarrow
{\mathcal P}$. Now Lemma~\ref{M+N} applies and we conclude that $\mathcal{V=M
\vee N}$ where $\mathcal M$ is generated by a monoid, while $\mathcal N$ is a
nil-variety. Substituting 1 for $y$ and $z$ in the identity~\eqref{u=s}, we
have that the identity $xq=qx$ holds in $\mathcal M$, so $\mathcal M$ is
generated by a commutative monoid.

\smallskip

\emph{Case}~2: at least one of the words $u,v,s$, and $t$ is $\mathcal
X$-stable. Suppose that this word is $u$. Denote its $\mathcal X$-image by
$U$. There is a non-trivial automorphism $\xi$ on $\F_{\{x,y,z,q\}}(\mathcal
X)$ with $\xi(U)\equiv U$. This automorphism performs some permutation on the
set $\{x,y,z,q\}$. Since $\ell_y(U)>\ell_z(U)>\ell_x(U)=\ell_q(U)$, we have
$\xi(y)\equiv y$ and $\xi(z)\equiv z$. Since the automorphism $\xi$ is
non-trivial, we have $\xi(x)\equiv q$ and $\xi(q)\equiv x$. Thus the
$\mathcal X$-images of the words $xy^{n+4m+3}z^nq$ and $qy^{n+4m+3}z^nx$
coincide. This means that the identity
\begin{equation}
\label{case 2 eq}
xy^{n+4m+3}z^nq=qy^{n+4m+3}z^nx
\end{equation}
holds in $\mathcal X$ and therefore, in $\mathcal V$. Lemma~\ref
{word problem} and its dual imply that this identity fails in the varieties
$\mathcal{LZ,RZ,P}$, and $\overleftarrow{\mathcal P}$. Hence we may apply
Lemma~\ref{M+N} and conclude that $\mathcal{V=M\vee N}$ where $\mathcal M$ is
generated by a monoid, while $\mathcal N$ is a nil-variety. Substituting 1
for $y$ and $z$ in the identity~\eqref{case 2 eq}, we have that the identity
$xq=qx$ holds in $\mathcal M$, so $\mathcal M$ is generated by a commutative
monoid. Analogous arguments may be used if one of the words $v,s$ or $t$ is
$\mathcal X$-stable.

\smallskip

It remains to verify that $\mathcal M$ is one of the varieties $\mathcal T$
or $\mathcal{SL}$. To do this, we note that the variety $\mathcal V$
satisfies the identities $u'=v'$ and $s'=t'$ where
\begin{align*}
&u'\equiv x^{n+2m+4}y^{n+2}z,\phantom{{}^2}\quad v'\equiv x^{n+m+4}y^{n+m+2}
z,\\
&s'\equiv x^{n+2m+3}y^{n+2}z^2,\quad t'\equiv x^{n+m+3}y^{n+m+2}z^2\ldotp
\end{align*}
These words have the same length and the same content. Besides that, these
words are pairwise $\mathcal X$-non-equivalent because if $w_1,w_2\in\{u',v',
s',t'\}$ and $w_1\not\equiv w_2$ then either $\max\limits_{a\in c(w_1)}
\{\ell_a (w_1)\}\ne\max\limits_{a\in c(w_2)}\{\ell_a(w_2)\}$ or $\min
\limits_{a\in c(w_1)}\{\ell_a(w_1)\}\ne\min\limits_{a\in c(w_2)}\{\ell_a
(w_2)\}$. Finally, the words $u',v',s'$, and $t'$ are $\mathcal X$-unstable
because $\ell_x(w)>\ell_y(w)>\ell_z(w)$ for any $w\in\{u',v',s',t'\}$. By
Lemma~\ref{u=v,s=t->u=s} the variety $\mathcal V$ satisfies the identity $u'=
s'$, that is, the identity
$$x^{n+2m+4}y^{n+2}z=x^{n+2m+3}y^{n+2}z^2\ldotp$$
Substituting 1 for $x$ and $y$ in this identity, we have that the identity
$z=z^2$ holds in $\mathcal M$. Thus $\mathcal{M\subseteq SL}$, whence
$\mathcal M$ is one of the varieties $\mathcal T$ or $\mathcal{SL}$.
\end{proof}

Now we are well prepared to complete the proofs of Theorems~\ref{Com mod nec}
and \ref{Com lmod} and to give new proofs of Propositions~\ref{SEM mod nec}
and~\ref{SEM lmod}.

\smallskip

Proposition~\ref{SEM mod nec} [Theorem~\ref{Com mod nec}] is directly implied
by Proposition~\ref{periodic in oc} with $\mathcal{X=SEM}$ [respectively
$\mathcal{X=COM}$] and Corollary~\ref{mod and lmod are periodic} [Lemma~\ref
{periodic or oc}].\qed

\smallskip

\emph{Proof of Proposition}~\ref{SEM lmod} \emph{and Theorem}~\ref{Com lmod}.
In both the statements, sufficiency immediately follows from Lemmas~\ref
{lmod nil} and~\ref{join with SL}. One can prove necessity. Let $\mathcal V$
be a [commutative] lower-modular [in \textbf{Com}] semigroup variety with
$\mathcal{V\ne SEM}$ [respectively $\mathcal{V\ne COM}$]. Then the variety
$\mathcal V$ is periodic by Corollary~\ref{mod and lmod are periodic}
[Lemma~\ref{periodic or oc}]. Applying Proposition~\ref{periodic in oc} with
$\mathcal{X=SEM}$ [respectively $\mathcal{X=COM}$], we conclude that
$\mathcal{V=M\vee N}$ where $\mathcal M$ is one of the varieties $\mathcal T$
or $\mathcal{SL}$, while $\mathcal N$ is a nil-variety. Lemma~\ref
{join with SL} implies that the variety $\mathcal N$ is lower-modular [in
\textbf{Com}]. Then the variety $\mathcal N$ is 0-reduced [in \textbf{Com}]
by Lemma~\ref{lmod nil}.\qed

\smallskip

To prove Theorem~\ref{Com mod nil-nec}, we need some additional auxiliary
results.

\begin{lemma}
\label{u=0->v=0}
An identity $v=0$ follows from an identity system $\{xy=yx,u=0\}$ if and only
if there is some \textup(possibly empty\textup) word $w$ and some
endomorphism $\xi$ on $\F$ such that the identity $v=w\xi(u)$ is balanced.
\end{lemma}

\begin{proof}
Let $U$ be the $\mathcal{COM}$-image of the word $u$. Consider the set
$$\mathfrak W=\{W\zeta(U)\mid W\in\F(\mathcal{COM}),\,\zeta\in\End(\F
(\mathcal{COM}))\}\ldotp$$
It is evident that $\mathfrak W$ is an ideal of the semigroup $\F(\mathcal
{COM})$ and that $\zeta(\mathfrak W)\subseteq\mathfrak W$ for any
$\zeta\in\End(\F(\mathcal{COM}))$. Hence the relation $\alpha$ on $\F
(\mathcal{COM})$ defined by the rule
$$W_1\,\alpha\,W_2\ \text{if and only if either}\ W_1\equiv W_2\ \text{or}\
W_1,W_2\in\mathfrak W$$
is a fully invariant congruence on $\F(\mathcal{COM})$. Put $\mathcal V=\var
\{xy=yx,\,u=0\}$ and consider the fully invariant congruence $\alpha'$
corresponding to $\mathcal V$. We shall prove that $\alpha=\alpha'$. Since
$\alpha'$ corresponds to a 0-reduced in \textbf{Com} variety, it has only one
non-singleton class $\mathfrak W'$. The set $\mathfrak W'$ is an ideal of the
semigroup $\F(\mathcal{COM})$. This ideal contains $U$ and satisfies $\zeta
(\mathfrak W')\subseteq\mathfrak W'$ for any $\zeta\in\End(\F(\mathcal
{COM}))$ because the congruence $\alpha'$ is fully invariant. Therefore
$\zeta(U)\in\mathfrak W'$ for any endomorphism $\zeta$. Hence $W\zeta(U)\in
\mathfrak W'$ for any $W\in\F(\mathcal{COM})$ because $\mathfrak W'$ is an
ideal. We see that $\mathfrak{W\subseteq W'}$, whence $\alpha\subseteq
\alpha'$. Further, $\mathfrak W$ is a zero of the factor semigroup $\F
(\mathcal{COM})/\alpha$. Since the inclusion $\zeta(U)\in\mathfrak W$ holds
for any endomorphism $\zeta$ on $\F(\mathcal{COM})$, the identity $u=0$ holds
in $\F(\mathcal{COM})/\alpha$. Hence $\F(\mathcal{COM})/\alpha\in\mathcal V$,
so $\alpha'\subseteq\alpha$. We have proved that $\alpha=\alpha'$ and hence
$\mathfrak{W=W'}$.

An identity $v=0$ follows from the system $\{xy=yx,u=0\}$ if and only if it
holds in $\mathcal V$. This is so if and only if the $\mathcal{COM}$-image
$V$ of the word $v$ belongs to $\mathfrak{W'=W}$, i.~e.\ $V\equiv W\zeta(U)$
for some $W\in\F(\mathcal{COM})$ and some $\zeta\in\End(\F(\mathcal{COM}))$.
But every element of $\F(\mathcal{COM})$ is a $\mathcal{COM}$-image of some
word and every endomorphism on $\F(\mathcal{COM})$ has the form $\varphi\xi$
where $\xi$ is an endomorphism on $\F$, while $\varphi$ is the natural
homomorphism from $\F$ to $\F(\mathcal{COM})$. Therefore the equality $V
\equiv W\zeta(U)$ is the equality of $\mathcal{COM}$-images of words $v$ and
$w\xi(u)$ for some word $w$ and some $\xi\in\End(\F)$. So this equality means
that the identity $v=w\zeta(u)$ holds in $\mathcal{COM}$, whence it is
balanced by Lemma~\ref{oc ident}.
\end{proof}

\begin{corollary}
\label{equivalence}
Words $u$ and $v$ are $\mathcal{COM}$-equivalent if and only if the identity
systems $\{xy=yx,u=0\}$ and $\{xy=yx,v=0\}$ are equivalent.
\end{corollary}

\begin{proof}
\emph{Necessity}. If $u$ and $v$ are $\mathcal{COM}$-equivalent then there is
some automorphism $\xi$ on $\F$ such that the identity $v=\xi(u)$ holds in
$\mathcal{COM}$, whence it is balanced by Lemma~\ref{oc ident}. Then the
identity $v=0$ follows from the system $\{xy=yx,u=0\}$ by Lemma~\ref
{u=0->v=0}. By symmetry, the identity $u=0$ follows from the system $\{xy=yx,
v=0\}$. Hence these two systems are equivalent.

\emph{Sufficiency}. Suppose that the systems $\{xy=yx,u=0\}$ and $\{xy=yx,v=
0\}$ are equivalent. By Lemma~\ref{u=0->v=0} there are some (possibly empty)
words $a$ and $b$ and some automorphisms $\xi$ and $\zeta$ on $\F$ such that
the identities $u=a\xi(v)$ and $v=b\zeta(u)$ are balanced. Then the identity
$u=a\xi(b\zeta(u))$ is balanced too. Hence $\ell(u)=\ell(a\xi(b\zeta(u)))$.
But this is possible only if the words $a$ and $b$ are empty and the
endomorphism $\xi\zeta$ maps every letter of $u$ to a letter. Further, $c(u)=
c(\xi\zeta(u))$ because the identity $u=\xi\zeta(u)$ is balanced. This means
that $\xi\zeta$ maps distinct letters of $u$ to distinct letters. Therefore
both endomorphisms $\xi$ and $\zeta$ have the same property and the
restrictions of $\xi$ and $\zeta$ on the semigroups $\F_{c(v)}$ and
$F_{c(u)}$ respectively are isomorphisms between these semigroups. Thus the
identity $u=\xi(v)$ is balanced and $\xi$ perfoms an isomorphism from
$\F_{c(v)}$ to $\F_{c(u)}$. Therefore $u$ is $\mathcal{COM}$-equivalent to
$v$.
\end{proof}

\emph{Proof of Theorem}~\ref{Com mod nil-nec}. Suppose that a commutative
nil-variety $\mathcal V$ is modular in \textbf{Com} and consider any identity
$u=v$ satisfied in $\mathcal V$. Suppose that the identities $u=0$ and $v=0$
are not satisfied in $\mathcal V$. We aim to prove that the identity $u=v$
follows from the commutative law and some substitutive identity satisfied in
$\mathcal V$.

By Lemma~\ref{V+G=ZR(V)+G} there is an Abelian periodic group variety
$\mathcal G$ such that
\begin{equation}
\label{G+V=G+ZR(V)} \mathcal{G\vee V=G\vee\ZR(V)}\ldotp
\end{equation}
Put $\mathcal X=\var\{xy=yx,\,u=0\}$, $\mathcal Y=\var\{xy=yx,\,v=0\}$, and
$\mathcal{Z=X\vee G}$. The variety $\mathcal G$ satisfies the identity $x^ny=
y$ for some $n$ and therefore the identity $x^nu=u$. This identity holds also
in $\mathcal X$, so it holds in $\mathcal Z$ and therefore in $\mathcal{V
\wedge Z}$. Being a nil-variety, $\mathcal{V\wedge Z}$ satisfies the identity
$u=0$ by Lemma~\ref{split}(ii). The identity $u=v$ holds in $\mathcal V$,
whence $u=v=0$ holds in $\mathcal{V\wedge Z}$, i.~e.
\begin{equation}
\label{inclusion}
\mathcal{V\wedge Z\subseteq Y}\ldotp
\end{equation}
So we have
\begin{align*}
\mathcal{(\ZR(V)\wedge X)\vee G}&\subseteq\mathcal{(\ZR(V)\vee G)\wedge(X\vee
G)}&&\\
&\mathcal{=(\ZR(V)\vee G)\wedge Z}&&\\
&\mathcal{=(V\vee G)\wedge Z}&&\text{by \eqref{G+V=G+ZR(V)}}\\
&\mathcal{=(V\wedge Z)\vee G}&&\text{because $\mathcal V$ is modular}\\
&&&\text{in \textbf{Com} and}\ \mathcal{G\subseteq Z}\\
&\mathcal{\subseteq Y\vee G}&&\text{by \eqref{inclusion}}\ldotp
\end{align*}
Thus $\mathcal{(\ZR(V)\wedge X)\vee G\subseteq Y\vee G}$. The identity $v=v
x^n$ holds in $\mathcal{Y\vee G}$, so it holds in $\mathcal{(\ZR(V)\wedge X)
\vee G}$ and therefore in $\mathcal{\ZR(V)\wedge X}$. Since $\mathcal{\ZR(V)
\wedge X}$ is a nil-variety, it satisfies the identity $v=0$ by Lemma~\ref
{split}(ii). Hence there is a deduction of this identity from identities of
the varieties $\ZR(\mathcal V)$ and $\mathcal X$, that is, a sequence of
words $w_0,w_1,\dots,w_n$ such that $v\equiv w_0$ and each of the identities
$w_0=w_1$, $w_1=w_2$, $w_{n-1}=w_n$, and $w_n=0$ holds in one of the
varieties $\ZR(\mathcal V)$ or $\mathcal X$. We may assume without loss of
generality that $w_0,w_1,\dots,w_n$ is the shortest sequence with these
properties. In particular, for any $i=0,1,\dots,n-1$, none of the varieties
$\ZR(\mathcal V)$ and $\mathcal X$ satisfies the identity $w_i=0$, and the
identity $w_i=w_{i+1}$ does not hold in both the varieties $\ZR(\mathcal V)$
and $\mathcal X$ simultaneously. Suppose that $n>0$. Let $i\in\{0,1,\dots,
n-1\}$. Since the varieties $\ZR(\mathcal V)$ and $\mathcal X$ are 0-reduced
in $\mathbf{Com}$ and none of these varieties satisfies the identity $w_i=0$,
the identity $w_i=w_{i+1}$ holds in $\mathcal{COM}$ and therefore, in both
the varieties $\ZR(\mathcal V)$ and $\mathcal X$. A contradiction shows that
$n=0$. Thus the identity $v=0$ holds in one of the varieties $\ZR(\mathcal
V)$ or $\mathcal X$. But this identity fails in $\ZR(\mathcal V)$ because it
fails in $\mathcal V$. So $v=0$ holds in $\mathcal X$. Analogously,
considering the variety $\mathcal{Y\vee G}$ rather than $\mathcal{X\vee G}$,
we can prove that the identity $u=0$ holds in $\mathcal Y$. This means that
$\mathcal{X=Y}$, so the identity systems $\{xy=yx,u=0\}$ and $\{xy=yx,v=0\}$
are equivalent. By Lemma~\ref{equivalence} the words $u$ and $v$ are
$\mathcal{COM}$-equivalent. Hence there is some automorphism $\xi$ on $\F
(\mathcal{COM})$ such that the identity $u=\xi(v)$ holds in $\mathcal{COM}$.
In particular, this identity is balanced (by Lemma~\ref{oc ident}) and holds
in $\mathcal V$. The identity $v=\xi(v)$ holds in $\mathcal V$ because the
identities $u=v$ and $u=\xi(v)$ hold in $\mathcal V$. If $c(v)\ne c(\xi(v))$
then $\mathcal V$ satisfies the identity $v=0$ by Lemma~\ref{split}(i). But
this is not true. Therefore $c(v)=c(\xi(v))$. This means that the identity
$v=\xi(v)$ is substitutive. Being balanced, the identity $u=\xi(v)$ follows
from the commutative law. The identities $u=\xi(v)$ and $v=\xi(v)$ imply $u=
v$. Therefore, $u=v$ follows from $v=\xi(v)$ and the commutative law. Since
the identity $v=\xi(v)$ is substitutive and holds in $\mathcal V$, we are
done.\qed

\smallskip

Theorem~\ref{Com mod nil-nec} and the `commutative half' of Lemma~\ref
{0-red is mod} provide a necessary and a sufficient condition for a
commutative nil-variety to be modular in \textbf{Com} respectively. The gap
between these conditions seems to be not very large. But the necessary
condition is not a sufficient one, while the sufficient condition is not a
necessary one. Indeed, it may be checked that the variety $\var\{xyzt=x^3=
0,\,x^2y=y^2x,\,xy=yx\}$ is modular in \textbf{Com} although it is not
0-reduced in \textbf{Com}, while the variety $\var\{x^5=0,\,x^3y^2=y^3x^2,\,x
y=yx\}$ is not modular in \textbf{Com} although it is given within $\mathcal
{COM}$ by 0-reduced and substitutive identities only.

\section{Corollaries}
\label{corol}

It was verified in~\cite{Shaprynskii-Vernikov-lmod} that a lower-modular
semigroup variety is modular. Theorem~\ref{Com lmod}, together with results
of~\cite{Shaprynskii-dn}, implies the following `commutative analog' of this
fact.

\begin{corollary}
\label{lmod is mod}
Every lower-modular in $\mathbf{Com}$ variety is modular in $\mathbf{Com}$.
\end{corollary}

\begin{proof}
Let $\mathcal V$ be a lower-modular in $\mathbf{Com}$ variety. We may assume
that $\mathcal{V\ne COM}$. Then Theorem~\ref{Com lmod} implies that $\mathcal
{V=M\vee N}$ where $\mathcal M$ is one of the varieties $\mathcal T$ or
$\mathcal{SL}$, while $\mathcal N$ is a 0-reduced in $\mathbf{Com}$ variety.
The variety $\mathcal N$ is modular in $\mathbf{Com}$ by Lemma~\ref
{0-red is mod}. It remains to refer to Lemma~\ref{join with SL}.
\end{proof}

By the way, we note that all five other possible implications between
properties of being a modular in \textbf{Com} variety, a lower-modular in
\textbf{Com} variety and an upper-modular in \textbf{Com} variety are fail.
For instance:
\begin{itemize}
\item[$\bullet$]the variety $\var\{xyzt=x^3=0,\,x^2y=y^2x,\,xy=yx\}$ is
modular in \textbf{Com} (as we have already mentioned at the end of
Section~\ref{proof}) but not lower-modular in \textbf{Com} (by Theorem~\ref
{Com lmod});
\item[$\bullet$]the variety $\var\{x^3=0,\,xy=yx\}$ is modular in \textbf
{Com} (by Lemma~\ref{0-red is mod}) and lower-modular in \textbf{Com} (by
Lemma~\ref{lmod nil}) but not upper-modular in \textbf{Com} (by Lemma~\ref
{umod 0-red});
\item[$\bullet$]the variety of Abelian groups of a prime exponent is
upper-modular in \textbf{Com} (because this variety is an atom of the lattice
\textbf{Com}) but neither modular in \textbf{Com} (by Theorem~\ref
{Com mod nec}) nor lower-modular in \textbf{Com} (by Theorem~\ref{Com lmod}).
\end{itemize}

It was proved in~\cite{Shaprynskii-dn}, Theorem~1.2, that a commutative
semigroup variety is neutral in $\mathbf{Com}$ if and only if it is both
distributive and codistributive in $\mathbf{Com}$. This assertion is
generalized by the following

\begin{corollary}
\label{lmod and umod is neut}
For a commutative semigroup variety $\mathcal V$, the following are
equivalent:
\begin{itemize}
\item[\textup{a)}]$\mathcal V$ is both lower-modular and upper-modular in
$\mathbf{Com}$;
\item[\textup{b)}]$\mathcal V$ is neutral in $\mathbf{Com}$;
\item[\textup{c)}]either $\mathcal{V=COM}$ or $\mathcal{V=M\vee N}$ where
$\mathcal M$ is one of the varieties $\mathcal T$ or $\mathcal{SL}$, while
$\mathcal N$ satisfies the identities $x^2y=0$ and $xy=yx$.
\end{itemize}
\end{corollary}

\begin{proof}
The equivalence of the statements~b) and~c) is proved in~\cite
{Shaprynskii-dn}, Theorem~1.2, while the implication~b)~$\longrightarrow$~a)
is evident.

a)~$\longrightarrow$~c). Suppose that a variety $\mathcal V$ is both
lower-modular and upper-modular in $\mathbf{Com}$. We may assume that
$\mathcal{V\ne COM}$. Then Theorem~\ref{Com lmod} implies that $\mathcal{V=M
\vee N}$ where $\mathcal M$ is one of the varieties $\mathcal T$ or $\mathcal
{SL}$, while $\mathcal N$ is a 0-reduced in $\mathbf{Com}$ variety. The
variety $\mathcal N$ is upper-modular in $\mathbf{Com}$ by Lemma~\ref
{join with SL}. It remains to refer to Lemma~\ref{umod 0-red}.
\end{proof}

Note that the analog of Corollary~\ref{lmod and umod is neut} for the lattice
\textbf{SEM} also is true (\cite{Vernikov-08}, Corollary~3.5).


\begin{thebibliography}{99}
\bibitem{Burris-Nelson-71}
\textsf{S.\,Burris} and \textsf{E.Nelson}, Embedding the dual of $\Pi_m$ in
the lattice of equational classes of commutative semigroups, \emph{Proc.\
Amer.\ Math.\ Soc.}, \textbf{30} (1971),
37--39.
\bibitem{Golubov-Sapir-82}
\textsf{E.\,A.\,Golubov} and \textsf{M.\,V.\,Sapir}, Residually small
varieties of semigroups, \emph{Izvestiya VUZ. Matematika}, No.\,11 (1982),
21--29, in Russian; Engl.\ translation: \emph{Soviet Math. Izv.\ VUZ},
\textbf{26}, No.\,11 (1982), 25--36.
\bibitem{Gratzer-98}
\textsf{G.\,Gr\"atzer}, General Lattice Theory, 2-nd ed., Birkhauser Verlag,
Basel (1998).
\bibitem{Jezek-81}
\textsf{J.~Je\v{z}ek}, The lattice of equational theories. Part~I: modular
elements, Czechosl.\ Math.\ J., \textbf{31} (1981), 127--152.
\bibitem{Jezek-McKenzie-93}
\textsf{J.\,Je\v{z}ek} and \textsf{R.\,N.\,McKenzie}, Definability in the
lattice of equational theories of semigroups, \emph{Semigroup Forum}, \textbf
{46} (1993), 199--245.
\bibitem{Kisielewicz-94}
\textsf{A.\,Kisielewicz}, Varieties of commutative semigroups, \emph{Trans.
Amer. Math. Soc.}, \textbf{342} (1994), 275--306.
\bibitem{Perkins-69}
\textsf{P.\,Perkins}, Bases for equational theories of semigroups, \emph{J.
Algebra}, \textbf{11} (1969), 298--314.
\bibitem{Pudlak-Tuma-80}
\textsf{P.\,Pudl\'ak} and \textsf{J.\,T\.uma}, Every finite lattice can be
embedded in a finite partition lattice, \emph{Algebra Universalis}, \textbf
{10} (1980), 74--95.
\bibitem{Shaprynskii-dn}
\textsf{V.\,Yu.\,Shaprynski\v{\i}}, Distributive and neutral elements of the
lattice of commutative semigroup varieties, \emph{Izvestiya VUZ. Matematika},
accepted, in Russian.
\bibitem{Shaprynskii-Vernikov-lmod}
\textsf{V.\,Yu.\,Shaprynski\v{\i}} and \textsf{B.\,M.\,Vernikov},
Lower-modular elements of the lattice of semigroup varieties.~III, \emph
{Acta. Sci. Math.} (\emph{Szeged}), accepted.
\bibitem{Shevrin-Vernikov-Volkov-09}
\textsf{L.\,N.\,Shevrin}, \textsf{B.\,M.\,Vernikov}, and \textsf
{M.\,V.\,Volkov}, Lattices of semigroup varieties, \emph{Izvestiya VUZ.
Matematika}, No.\,3 (2009), 3--36, in Russian; Engl.\ translation: \emph
{Russian Math. Izv.\ VUZ}, \textbf{53}, No.\,3 (2009), 1--28.
\bibitem{Vernikov-07-mod}
\textsf{B.\,M.\,Vernikov}, On modular elements of the lattice of semigroup
varieties, \emph{Comment.\ Math.\ Univ.\ Carol.}, \textbf{48} (2007),
595--606.
\bibitem{Vernikov-07-lmod}
\textsf{B.\,M.\,Vernikov}, Lower-modular elements of the lattice of semigroup
varieties, \emph{Semigroup Forum}, \textbf{75} (2007), 554--566.
\bibitem{Vernikov-08}
\textsf{B.\,M.\,Vernikov}, Lower-modular elements of the lattice of semigroup
varieties.~II, \emph{Acta Sci.\ Math.}\ (\emph{Szeged}), \textbf{74} (2008),
539--556.
\bibitem{Vernikov-Shaprynskii-10}
\textsf{B.\,M.\,Vernikov} and \textsf{V.\,Yu.\,Shaprynski\v{\i}},
Distributive elements of the lattice of semigroup varieties,
\emph{Algebra i Logika}, \textbf{49} (2010), 303--330, in Russian; Engl.\
translation: \emph{Algebra and Logic}, \textbf{49} (2010), 201--220.
\bibitem{Vernikov-Volkov-88}
\textsf{B.\,M.\,Vernikov} and \textsf{M.\,V.\,Volkov}, Lattices of nilpotent
semigroup varieties, \emph{Algebraicheskie Systemy i Ih Mnogoobraziya} (\emph
{Algebraic Systems and Their Varieties}), Sverdlovsk: Ural State University
(1988), 53--65, in Russian.
\bibitem{Volkov-89}
\textsf{M.\,V.\,Volkov}, Semigroup varieties with modular subvariety
lattices, \emph{Izvestiya VUZ. Matematika}, No.\,6 (1989), 51--60, in
Russian; Engl.\ translation: \emph{Soviet Math. Izv.\ VUZ}, \textbf{33},
No.\,6 (1989), 48--58.
\bibitem{Volkov-05}
\textsf{M.\,V.\,Volkov}, Modular elements of the lattice of semigroup
varieties, \emph{Contrib.\ General Algebra}, \textbf{16} (2005), 275--288.
\end{thebibliography}
\end{document}